\documentclass{elsarticle}

\usepackage[english]{babel}
\usepackage{dsfont} 
\usepackage{graphicx}
\usepackage{color}
\usepackage{hyperref}
\usepackage{enumerate}
\usepackage{amssymb,empheq} 
\usepackage{amsthm} 
\usepackage{cleveref}
\usepackage{xcolor}

\allowdisplaybreaks
\newtheorem{thm}{Theorem}[section]
\newtheorem{mthm}[thm]{Main Theorem}
\newtheorem{lem}[thm]{Lemma}

\newtheorem{cor}[thm]{Corollary}

\newtheorem{open}[thm]{Open Problem}

\theoremstyle{definition}

\newtheorem{dfn}[thm]{Definition}
\newtheorem{exm}[thm]{Example}
\newtheorem{exms}[thm]{Examples}

\theoremstyle{remark}
\newtheorem{rem}[thm]{Remark}

\newcommand{\exmsymbol}{\hfill$\circ$}

\newcommand{\cset}{\mathds{C}}

\newcommand{\nset}{\mathds{N}}

\newcommand{\qset}{\mathds{Q}}
\newcommand{\rset}{\mathds{R}}

\newcommand{\diff}{\mathrm{d}}

\newcommand{\pos}{\mathrm{Pos}}

\newcommand{\supp}{\mathrm{supp}\,}

\newcommand{\dom}{\mathrm{dom}\,}
\newcommand{\id}{\mathrm{id}}
\newcommand{\one}{\mathds{1}}

\newcommand{\cS}{\mathcal{S}}

\newcommand{\cV}{\mathcal{V}}

\newcommand{\fB}{\mathfrak{B}}

\newcommand{\fD}{\mathfrak{D}}
\newcommand{\fd}{\mathfrak{d}}

\newcommand{\fg}{\mathfrak{g}}

\newcommand{\fp}{\mathfrak{p}}

\newcommand{\fs}{\mathfrak{s}}
\newcommand{\fS}{\mathfrak{S}}

\author{Philipp J.\ di Dio}
\address{Department of Mathematics and Statistics, University of Konstanz, Universit\"atsstra{\ss}e 10, D-78464 Konstanz, Germany}
\address{Zukunftskolleg, Universtity of Konstanz, Universit\"atsstra{\ss}e 10, D-78464 Konstanz, Germany}
\address{philipp.didio@uni-konstanz.de}

\journal{ArXiv}

\title{On Positivity Preservers with constant Coefficients and their Generators}

\begin{document}

\begin{abstract}
In this work we study positivity preservers $T:\rset[x_1,\dots,x_n]\to\rset[x_1,\dots,x_n]$ with constant coefficients and define their generators $A$ if they exist, i.e., $\exp(A) = T$.
We use the theory of regular Fr\'echet Lie groups to show the first main result.
A positivity preserver with constant coefficients has a generator if and only if it is represented by an infinitely divisible measure (\Cref{thm:infinitelyDivisibleMeasure}).
In the second main result (\Cref{thm:mainPosGenerators}) we use the L\'evy--Khinchin formula to fully characterize the generators of positivity preservers with constant coefficients.
\end{abstract}


\maketitle



\section{Introduction}

Non-negative polynomials are widely applied and studied.
A first step to investigate non-negative polynomials is to study the convex cone
\[\pos(K) := \{p\in\rset[x_1,\dots,x_n] \,|\, p(x)\geq 0\ \text{for all}\ x\in K\}\]
of non-negative polynomials on some $K\subseteq\rset^n$.

The second logical step is to study (linear) maps
\[T:\rset[x_1,\dots,x_n]\to\rset[x_1,\dots,x_n]\]
between polynomials, especially when they preserve non-negativity, i.e.,
\[T\pos(K)\subseteq\pos(K),\]
see e.g.\ \cite{guterman08,netzer10,borcea11}.
With the multi-index notation $\partial^\alpha = \partial_1^{\alpha_1}\cdots \partial_n^{\alpha_n}$ for all $\alpha = (\alpha_1,\dots,\alpha_n)\in\nset_0^n$ with $n\in\nset$ the following is long known.
An explicit proof can e.g.\ be found in \cite{netzer10}.

\begin{lem}[folklore, see e.g.\ {\cite[Lem.\ 2.3]{netzer10}}]\label{lem:folklore}
Let $n\in\nset_0$ and let
\[T:\rset[x_1,\dots,x_n]\to\rset[x_1,\dots,x_n]\]
be linear.
Then for all $\alpha\in\nset_0^n$ there exist unique $q_\alpha\in\rset[x_1,\dots,x_n]$ such that
\[ T = \sum_{\alpha\in\nset_0^n} q_\alpha\cdot \partial^\alpha.\]
\end{lem}

The map $T:\rset[x_1,\dots,x_n]\to\rset[x_1,\dots,x_n]$ is called to have constant coefficients if $q_\alpha\in\rset$ for all $\alpha\in\nset_0^n$.
The subset with $q_0=1$ is denoted by
\[\fD := \left\{\sum_{\alpha\in\nset_0^n} q_\alpha\cdot \partial^\alpha \,\middle|\, q_\alpha\in\rset,\ q_0 = 1\right\}\quad \subsetneq \rset[[\partial_1,\dots,\partial_n]].\]
Positivity preservers are fully described by their polynomial coefficients. We have the following.

\begin{thm}[see e.g.\ {\cite[Thm.\ 3.1]{borcea11}}]\label{thm:posPresChar}
Let $n\in\nset$ and let $T:\rset[x_1,\dots,x_n]\to\rset[x_1,\dots,x_n]$ be linear.
Then the following are equivalent:
\begin{enumerate}[(i)]
\item $T\pos(\rset^n)\subseteq\pos(\rset^n)$.

\item $(\alpha!\cdot q_\alpha(y))_{\alpha\in\nset_0^n}$ is a moment sequence for all $y\in\rset^n$.
\end{enumerate}
\end{thm}

\begin{exm}\label{exm:heat}
Let $n=1$.
Then
\[\exp(t\cdot\partial_x^2) = \sum_{k\in\nset_0} \frac{t^k\cdot \partial_x^{2k}}{k!}\]
is a positivity preserver for all $t\geq 0$.
\exmsymbol
\end{exm}

The following is a linear operator $T$ which is not a positivity preserver.

\begin{exm}\label{exm:notpos}
Let $k\geq 3$ and $a\in\rset\setminus\{0\}$.
Then
\[\exp(a\partial_x^k) := \sum_{j\in\nset_0} \frac{a^j\cdot\partial_x^{j\cdot k}}{j!}\]
is not a positivity preserver since $q_{2k} = a^{2k} \neq 0$ but $q_{2k+2} = 0$, i.e., $(j!\cdot q_j)_{j\in\nset_0}$ is not a moment sequence.
\exmsymbol
\end{exm}

From elementary calculations and measure theory we get the following additional properties of positivity preservers.
Let $T$ be a positivity preserver with constant coefficients and $\mu$ be a representing measure of the corresponding moment sequence $s = (s_\alpha)_{\alpha\in\nset_0^n} = (\alpha!\cdot q_\alpha)_{\alpha\in\nset_0^n}$, then
\[(Tf)(x) = \int_{\rset^n} f(x+y)~\diff\mu(y)\]
for all $f\in\rset[x_1,\dots,x_n]$. For short we call $\mu$ also a \emph{representing measure of $T$}.
If follows for two such operators $T$ and $T'$ with representing measures $\mu$ and $\mu'$ that
\begin{equation}\label{eq:convolution}
(T T' f)(x) = \int f(x+y)~\diff(\mu * \mu')(y)
\end{equation}
where $\mu * \mu'$ is the \emph{convolution of the two measures} $\mu$ and $\mu'$:
\begin{multline}\label{eq:measureConv}
(\mu*\mu')(A) := \int_{\rset^n\times\rset^n} \chi_A(x+y)~\diff\mu(x)~\diff\mu'(y)\\
= \int_{\rset^n} \mu(A-y)~\diff\mu'(y) = \int_{\rset^n} \mu'(A-x)~\diff\mu(x)
\end{multline}
for any $A\in\fB(\rset^n)$, see e.g.\ \cite[Sect.\ 3.9]{bogachevMeasureTheory}.
Here, $\chi_A$ is the characteristic function of the Borel set $A$.
For the supports we have
\begin{equation}\label{eq:measureConvolutionSupport}
\supp (\mu*\mu') = \supp\mu + \supp\mu'.
\end{equation}
This can easily be proved from (\ref{eq:measureConv}) or found in the literature, see e.g.\ \cite[Prop.\ 14.5 (ii)]{choquet69} for a special case.
We abbreviate
\[\mu^{*k} := \underbrace{\mu*\dots*\mu}_{k\text{-times}}\]
for all $k\in\nset$ and set $\mu^{*0} := \delta_0$ with $\delta_0$ the Dirac measure centered at $0$.

Previously we investigated the heat semi-group (see \Cref{exm:heat} for $n=1$) and its action on (non-negative) polynomials, see \cite{curtoHeat22,curtoHeat23}.
For $n\in\nset$ we have that $\exp(t\Delta)p_0$ is the unique solution of the polynomial valued heat equation
\[\partial_t p = \Delta p\]
with initial values $p_0\in\rset[x_1,\dots,x_n]$.
In \cite{curtoHeat23} we observed the very strange behavior that several non-negative polynomials which are not sums of squares (e.g.\ the Motzkin and the Choi--Lam polynomial) become a sum of squares in finite time under the heat equation.
In \cite[Thm.\ 3.20]{curtoHeat23} we showed that every non-negative $p\in\rset[x,y,z]_{\leq 4}$ becomes a sum of squares in finite time.
All these observations hold for the heat equation, i.e., the family $(\exp(t\Delta))_{t\geq 0}$ of positivity preservers with constant coefficients with the generator $\Delta = \partial_1^2 + \dots + \partial_n^2$.
So a natural question to further study these effects is to understand any family $(\exp(tA))_{t\geq 0}$ of positivity preservers and hence to determine all possible generators $A$.
This is the main question of the current work for the constant coefficient case.
It is answered in \Cref{thm:mainPosGenerators}.

The paper is structured as follows.
To define and work with $\exp(A)$ we repeat in the preliminaries for the readers convenience the notion of Lie groups, Lie algebras, and especially their lesser known infinite dimensional versions of regular Fr\'echet Lie groups.
We will see that $\fD$ is a regular Fr\'echet Lie group with a Lie algebra $\fd$.
By \Cref{thm:posPresChar} we transport this property to sequences in \Cref{sec:sequences}.
\Cref{sec:generators} contains the two main results.
In the first \Cref{thm:infinitelyDivisibleMeasure} we show that a positivity preserver with constant coefficients has a generator if and only if it is represented by an infinitely divisible measure.
Using the L\'evy--Khinchin formula (which is also given in the preliminary, see \Cref{thm:leviKhinchinFormula}) we give in the second \Cref{thm:mainPosGenerators} the full description of all generators of positivity preservers with constant coefficients.
In \Cref{sec:strange} we discuss a strange action on non-negative polynomials on $[0,\infty)$ caused by the heat equation with Dirichlet boundary conditions.
We end this paper with a summary and an open question.

\section{Preliminaries}

Lie groups and Lie algebras are standard concepts in mathematics \cite{warner83}.
However, this only applies to the finite dimensional cases.
For the readers convenience we give here the infinite dimensional definitions and examples we need to make the paper as self-contained as possible.
For the sake of completeness we also include the explicit statement of the L\'evy--Khinchin formula in \Cref{thm:leviKhinchinFormula}.

\subsection{Lie Groups and their Lie Algebras}

\begin{dfn}
A group $(G,\,\cdot\,)$ is called a \emph{Lie group} if $G$ is also a $n$-dimensional smooth manifold, $n\in\nset$, such that $G\times G\to G: (A,B)\mapsto AB^{-1}$ is smooth.
The \emph{Lie algebra $\fg$ of $G$} is the tangent space $T_e G$ at the identity $e\in G$.
\end{dfn}

The connection between the Lie algebra $\fg$ and the Lie group $G$ is given by the \emph{exponential map} $\exp:\fg\to G$.
For the special case $G = \mathrm{Gl}(n,\cset)$ the exponential mapping fulfills the following.

\begin{lem}[see e.g.\ {\cite[p.\ 134, Ex.\ 15]{warner83}}]\label{lem:expLog}
Let $n\in\nset$.
The following hold:
\begin{enumerate}[(i)]
\item The exponential map
\[\exp:\mathrm{Gl}(n,\cset)\to \mathrm{Gl}(n,\cset),\quad A\mapsto \exp A := \sum_{k\in\nset_0} \frac{A^k}{k!}\]
is surjective.

\item Let $\id+N\in\mathrm{Gl}(n,\cset)$ be such that $N$ is nilpotent.
Then
\[\log (\id + N) = -\sum_{k\in\nset} \frac{(-N)^k}{k}\]
is well-defined and
\[\exp(\log (\id+N)) = \id+N\]
holds.
\end{enumerate}
\end{lem}

The proof of \Cref{lem:expLog} (ii) follows from formal power series calculations.
For more on Lie groups and Lie algebras see e.g.\ \cite{warner83}.

The charts $\varphi:U\subseteq\rset^n\to G$ of the $n$-dimensional smooth manifold $G$ induce the Euclidean topology on the group $G$ and $(A,B)\mapsto AB^{-1}$ is smooth with respect to this topology.
When extending this to infinite dimensions more than one topology is possible and the choice of topology on $G$ is important.

\subsection{Fr\'echet Spaces}

\begin{dfn}
A topological vector space $\cV$ is called a \emph{Fr\'echet space} if the following three conditions are fulfilled:
\begin{enumerate}[(i)]
\item $\cV$ is metrizable (i.e., $\cV$ is Hausdorff),
\item $\cV$ is complete, and
\item $\cV$ is locally convex.
\end{enumerate}
\end{dfn}

It is clear that $\rset[x_1,\dots,x_n]_{\leq d}$ is a Fr\'echet space for all $d\in\nset_0$ since they are finite dimensional.
Their Fr\'echet topology is unique.

\begin{exm}[see e.g.\ {\cite[pp.\ 91--92, Ex.\ III]{treves67}}]\label{exm:powerSeriesTop}
Let $n\in\nset$.
The vector space $\rset[[x_1,\dots,x_n]]$ of formal power series
\[p=\sum_{\alpha\in\nset_0^n} c_\alpha\cdot x^\alpha \qquad\text{with}\ c_\alpha\in\rset\]
equipped with the topology induced by the semi-norms
\[|p|_d := \sup_{|\alpha|\leq d} |c_\alpha| \qquad\text{with}\ d\in\nset_0\]
is a Fr\'echet space.
In other words we have the convergence
\[p_i = \sum_{\alpha\in\nset_0^n} c_{i,\alpha}\cdot x^\alpha \quad\to\quad p=\sum_{\alpha\in\nset_0} c_\alpha\cdot x^\alpha \qquad\text{for}\ i\to\infty\]
if and only if $c_{i,\alpha}\xrightarrow{i\to\infty} c_\alpha$ for all $\alpha\in\nset_0^n$.
\exmsymbol
\end{exm}

For more on topological vector spaces see e.g.\ \cite{treves67,schaefer99}.

\subsection{Regular Fr\'echet Lie Groups and their Lie Algebras}

Already Hideki Omori stated the following, see \cite[pp.\ III-IV]{omori74}:
\begin{quote}\itshape
[G]eneral Fr\'echet manifolds are very difficult to treat.
For instance, there are some difficulties in the definition of tangent bundles, hence in the definition of the concept of $C^\infty$-mappings.
Of course, there is neither an implicit function theorem nor a Frobenius theorem in general.
Thus, it is difficult to give a theory of general Fr\'echet Lie groups.
\end{quote}
A more detailed study is given by Omori in \cite{omori97} and the theory successfully evolved since then, see e.g.\ \cite{omori74,kac85,omori97,stras02,wurz04,schmed23} and references therein.
We will give here only the basic definitions which will be needed for our study.

For the definition of $C^1(G,\fg)$ see e.g.\ \cite[pp.\ 9--10]{omori97}.
Since $\rset[[\partial_1,\dots,\partial_n]]$ has the Fr\'echet topology in \Cref{exm:powerSeriesTop}, i.e., the coordinate-wise convergence, we have that for every $m\in\nset_0$ a function $F:\rset\to\rset[[\partial_1,\dots,\partial_n]]$, $F(t) = \sum_{\alpha\in\nset_0^n} F_\alpha(t)\cdot\partial^\alpha$ is $C^m$ if and only if every coordinate $F_\alpha$ is $C^m$.

\begin{dfn}[see e.g.\ {\cite[p.\ 63, Dfn.\ 1.1]{omori97}}]\label{dfn:frechetLieGroup}
We call $(G,\,\cdot\,)$ a (\emph{regular}) \emph{Fr\'echet Lie group} if the following conditions are fulfilled:
\begin{enumerate}[(i)]
\item $G$ is an infinite dimensional smooth Fr\'echet manifold.

\item $(G,\,\cdot\,)$ is a group.

\item The map $G\times G\to G$, $(A,B)\mapsto A\cdot B^{-1}$ is smooth.

\item The \emph{Fr\'echet Lie algebra} $\fg$ of $G$ is isomorphic to the tangent space $T_e G$ of $G$ at the unit element $e\in G$.

\item $\exp:\fg\to G$ is a smooth mapping such that
\[\left.\frac{\diff}{\diff t}\right|_{t=0} \exp(t u) = u\]
holds for all $u\in\fg$.

\item The space $C^1(G,\fg)$ of $C^1$-curves in $G$ coincides with the set of all $C^1$-curves in $G$ under the Fr\'echet topology.
\end{enumerate}
\end{dfn}

For more on infinite dimensional manifolds, differential calculus, Lie groups, and Lie algebras see e.g.\ \cite{leslie67,omori97,schmed23}.

\subsection{The Lie Group $\fD_d$}
\label{sec:liegroup}

\begin{dfn}\label{dfn:Dd}
Let $n\in\nset$ and $d\in\nset_0$. We define
$\fD_d := \fD|_{\rset[x_1,\dots,x_n]_{\leq d}}$.
\end{dfn}

Since
\[A\rset[x_1,\dots,x_n]_{\leq d}\subseteq\rset[x_1,\dots,x_n]_{\leq d}\]
holds for all $d\in\nset_0$ and $A\in\fD$ the $\fD_d$ are well-defined.
From \Cref{dfn:Dd} we see that $\fD_d$ consists only of operators of the form
\[\sum_{\alpha\in\nset_0^n: |\alpha|\leq d} c_\alpha\cdot\partial^\alpha\]
with $c_\alpha\in\rset$ and $c_0=1$
since on $\rset[x_1,\dots,x_n]_{\leq d}$ every operator $\partial^\beta$ with $|\beta|>d$ fulfills $\partial^\beta p = 0$ for all $p\in\rset[x_1,\dots,x_n]_{\leq d}$, i.e., $\partial^\beta = 0$ on $\rset[x_1,\dots,x_n]_{\leq d}$.

\begin{rem}
We can also define $\fD_d$ by $\fD/\langle \partial^\alpha\, |\, |\alpha|=d+1\rangle$.
Both definitions are almost identical.
However, \Cref{dfn:Dd} has the following advantage.
In $\fD/\langle \partial^\alpha\, |\, |\alpha|=d+1\rangle$ we have the problem that we are working with equivalence classes and hence we can not calculate $A+B$ for $A\in\fD_d$ and $B\in\fD_{e}$ for $d\neq e$.
With \Cref{dfn:Dd} we can calculate $A+B$ for $A\in\fD_d$ and $B\in\fD_e$ with $d\neq e$ since $A+B$ is defined on
\[\dom(A+B) = \dom A \cap \dom B = \rset[x_1,\dots,x_n]_{\leq\min(d,e)}\]
as usual for (unbounded) operators \cite{schmudUnbound}.
\Cref{dfn:Dd} can then even be used to calculate $A+B$ for $A$ on $\rset[x_1,\dots,x_n]_{\leq d}$ and $B$ on $\rset[x_1,\dots,x_m]_{\leq e}$ for $n\neq m$ and $d\neq e$ on
\[\dom(A+B)=\rset[x_1,\dots,x_{\min\{n,m\}}]_{\leq\min\{d,e\}}.\tag*{$\circ$}\]
\end{rem}

\begin{exm}\label{exm:n1d3}
Let $n=1$ and $d=3$. Then
\[\fD_3 = \left\{ 1 + c_1\partial_x + c_2\partial_x^2 + c_3\partial_x^3
\;\middle|\; c_1,c_2,c_3\in\rset \right\}\quad \text{on}\quad \rset[x]_{\leq 3} .\]
Let
\[A = 1 + a_1\partial_x + a_2\partial_x^2 + a_3\partial_x^3 \quad\text{and}\quad B = 1 + b_1\partial_x + b_2\partial_x^2 + b_3\partial_x^3\]
be in $\fD_3$.
Then
\begin{equation}\label{eq:d3mult}
\begin{split}
AB &= BA = (1 + a_1\partial_x + a_2\partial_x^2 + a_3\partial_x^3)\cdot (1 + b_1\partial_x + b_2\partial_x^2 + b_3\partial_x^3)\\
&= 1 + (a_1 + b_1)\partial_x + (a_2 + a_1 b_1 + b_2)\partial_x^2 + (a_3 + a_2 b_1 + a_1 b_2 + b_3)\partial_x^3
\end{split}
\end{equation}
since derivatives $\partial^i$ with $i\geq 4$ are the zero operators on $\rset[x]_{\leq 3}$.
Hence, $(\fD_3,\,\cdot\,)$ is a commutative semi-group with neutral element $\one = 1$.

We will now see that $\fD_3$ is even a commutative group.
For that it is sufficient to find for any $A\in\fD_3$ a $B\in\fD_3$ with $AB= \one$.
By (\ref{eq:d3mult}) $AB = \one$ is equivalent to
\begin{align*}
0 &= a_1 + b_1 &&\Rightarrow\quad b_1 = -a_1\\
0 &= a_2 + a_1 b_1 + b_2 &&\Rightarrow\quad b_2 = -a_2 + a_1^2\\
0 &= a_3 + a_2 b_1 + a_1 b_2 + b_3 &&\Rightarrow\quad b_3 = -a_3 + 2a_2 a_1 - a_1^3,
\end{align*}
i.e., every $A\in\fD_3$ has the unique inverse
\[A^{-1} = 1 - a_1\partial_x + (-a_2 + a_1^2)\partial_x^2 + (-a_3 + 2a_2 a_1 - a_1^3)\partial_x^3 \quad\in\fD_3.\]
Hence, $(\fD_3,\,\cdot\,)$ is a commutative group.
\exmsymbol
\end{exm}

We have seen in the previous example that $(\fD_d,\,\cdot\,)$ for $n=1$ and $d=3$ is a commutative group.
This holds for all $n\in\nset$ and $d\in\nset_0$.

\begin{lem}\label{lem:ddgroup}
Let $n\in\nset$ and $d\in\nset_0$. Then $(\fD_d,\,\cdot\,)$ is a commutative group.
\end{lem}
\begin{proof}
Let $A = \sum_{\alpha: |\alpha|\leq d} a_\alpha\partial^\alpha, B = \sum_{\beta: |\beta|\leq d} b_\beta\partial^\beta\in\fD_d$, i.e., $a_0 = b_0 =1$.
Then
\[AB = C = \sum_{\gamma: |\gamma|\leq d} c_\gamma\cdot\partial^\gamma\]
holds with
\begin{equation}\label{eq:inverseSystem}
c_\gamma = \sum_{\alpha,\beta\in\nset_0^n: \alpha + \beta = \gamma} a_\alpha b_\beta.
\end{equation}
Let $\alpha = (\alpha_1,\dots,\alpha_n)\succeq\beta = (\beta_1,\dots,\beta_n)$ on $\nset_0^n$ if and only if $\alpha_i \geq \beta_i$ for all $i=1,\dots,n$.
Then (\ref{eq:inverseSystem}) can be solved by induction on $|\gamma|$.
For $|\gamma|=0$ we have $c_0 = a_0\cdot b_0$ and $a_0 = c_0 = 1$, i.e., $b_0 = 1$.
So assume (\ref{eq:inverseSystem}) is solved for all $c_\gamma$ with $|\gamma|\leq d-1$.
Then for any $\gamma\in\nset_0$ with $|\gamma|=d$ we have
\begin{equation}\label{eq:inverseA}
b_\gamma = a_0 b_\gamma = -\sum_{\alpha\in\nset_0\setminus\{0\}: \gamma\succeq\alpha} a_\alpha\cdot b_{\gamma-\alpha},
\end{equation}
i.e., the system (\ref{eq:inverseSystem}) of equations has a unique solution gained by induction.
Hence, for every $A\in\fD_d$ there exists a unique $B\in\fD_d$ with $AB = BA = \one$.
\end{proof}

From \Cref{lem:ddgroup} we have seen that $(\fD_d,\,\cdot\,)$ for any $n\in\nset$ and $d\in\nset_0$ is a commutative group.
Let
\begin{equation}\label{eq:iota}
\iota_d: \{1\}\times\rset^{\binom{n+d}{n}-1}\to\fD_d,\quad (a_\alpha)_{\alpha\in\nset_0^n: |\alpha|\leq d} \mapsto \sum_{\alpha\in\nset_0^n: |\alpha|\leq d} a_\alpha\cdot\partial^\alpha
\end{equation}
be an affine linear map.
Then $\iota_d$ in (\ref{eq:iota}) a diffeomorphism and it is a coordinate map for $\fD_d$.
The smooth manifold $\{1\}\times\rset^{\binom{n+d}{n}-1}$ inherits the group structure of $\fD_d$ through $\iota_d$, i.e.,
\begin{equation}
(\fD_d,\,\cdot\,)\; \overset{\iota_d}{\cong}\; \left(\{1\}\times\rset^{\binom{n+d}{n}-1},\odot\right)
\end{equation}
Hence, the map $\iota_d$ shows the following.

\begin{thm}
Let $n\in\nset$ and $d\in\nset_0$. Then $(\fD_d,\,\cdot\,)$ is a Lie group.
\end{thm}
\begin{proof}
The map $\iota_d$ in (\ref{eq:iota}) is a diffeomorphism between $\fD_d$ and $\{1\}\times\rset^{\binom{n+d}{n}-1}$.
Hence, $\fD_d$ is a differentiable manifold which possesses the group structure $(\fD_d,\,\cdot\,)$.
By (\ref{eq:inverseSystem}) and (\ref{eq:inverseA}) we have that the map
\[\fD_d\times \fD_d\to\fD_d,\quad (A,B)\mapsto AB^{-1}\]
is smooth.
Hence, $(\fD_d,\,\cdot\,)$ is a commutative Lie group.
\end{proof}

\subsection{The Lie Algebra $\fd_d$ of $\fD_d$}
\label{sec:liealgebra}

Since every $A\in\fD_d$ is a linear map
\[A:\rset[x_1,\dots,x_n]_{\leq d}\to\rset[x_1,\dots,x_n]_{\leq d}\]
between finite-dimensional vector spaces we can choose a basis of $\rset[x_1,\dots,x_n]_{\leq d}$ and get a matrix representation $\tilde{A}$ of $A$.
Take the monomial basis of $\rset[x_1,\dots,x_n]_{\leq d}$.
Then $\tilde{A}$ is an upper triangular matrix with diagonal entries $1$.

\begin{exm}[\Cref{exm:n1d3} continued]\label{exm:n1d3cont}
Let $n=1$ and $d=3$.
Then every $A = 1 + a_1\partial_x + a_2\partial_x^2 + a_3\partial_x^3\in\fD_3$ has with the monomial basis $\{1,x,x^2,x^3\}$ of $\rset[x]_{\leq 3}$ the matrix representation
\[\tilde{A} = \begin{pmatrix}
1 & a_1 & 2a_2 & 6a_3\\
0 & 1 & 2a_1 & 6a_2\\
0 & 0 & 1 & 3a_1\\
0 & 0 & 0 & 1
\end{pmatrix}\]
and we therefore set
\[\tilde{\fD}_3 := \left\{ \begin{pmatrix}
1 & a_1 & 2a_2 & 6a_3\\
0 & 1 & 2a_1 & 6a_2\\
0 & 0 & 1 & 3a_1\\
0 & 0 & 0 & 1
\end{pmatrix} \,\middle|\, a_1,a_2,a_3\in\rset \right\}.\]
Hence, $(\tilde{A}-\id)^4 = 0$ as a matrix and also $(A-\one)^4 = 0$ as an operator on $\rset[x]_{\leq 3}$.
From \Cref{lem:expLog} we find that the matrix valued exponential map
\[\exp:\mathrm{gl}(4,\cset)\to\mathrm{Gl}(4,\cset),\quad \tilde{A}\mapsto \exp(\tilde{A}) := \sum_{k\in\nset_0} \frac{\tilde{A}^k}{k!}\]
is surjective and the logarithm
\[\tilde{A}\mapsto \log \tilde{A} := -\sum_{k\in\nset} \frac{(\id-\tilde{A})^k}{k}\]
is well-defined for all $\id + N\in\mathrm{Gl}(4,\cset)$ with $N$ nilpotent.
Since $\tilde{\fD}_3\subseteq \mathrm{Gl}(4,\cset)$ with $(\id-\tilde{A})^4 = 0$ for all $\tilde{A}\in\tilde{\fD}_3$ we have
\begin{equation}\label{eq:logA}
\log:\tilde{\fD}_3\to\mathrm{gl}(4,\cset), \tilde{A}\mapsto\log \tilde{A} = -\sum_{k=1}^3 \frac{(\id-\tilde{A})^k}{k}.
\end{equation}
Since also $(A-\one)^4 = 0$ for all $A\in\fD_3$ we can use (\ref{eq:logA}) also for the differential operators in $\fD_3$:
\[\log:\fD_3\to\left\{d_0 + d_1\partial_x + d_2\partial_x^2 + d_3\partial_x^3 \,\middle|\, d_0,\dots,d_3\in\rset\right\},\quad A\mapsto -\sum_{k=1}^3 \frac{(\one-A)^k}{k}.\]
To determine the image $\log\fD_3$ recall that also $\log$ is an injective map by \Cref{lem:expLog} and hence $\log\fD_3$ is $3$-dimensional with $d_0 = 0$, i.e., we have
\[\log\fD_3 = \left\{d_1\partial_x + d_2\partial_x^2 + d_3\partial_x^3 \,\middle|\, d_1,d_2,d_3\in\rset\right\} =: \fd_3.\]
In summary, since $A^4 = 0$ for all $A\in\fd_3$ we have that
\begin{equation}\label{eq:df4exp}
\exp:\fd_3\to\fD_3,\quad A\mapsto \sum_{k=0}^3 \frac{A^k}{k!}
\end{equation}
is surjective with inverse
\begin{equation}\label{eq:df4log}
\log:\fD_3\to\fd_3,\quad A\mapsto -\sum_{k=1}^3 \frac{(\one-A)^k}{k}.
\end{equation}
Therefore, $\fd_3$ is the Lie algebra of $\fD_3$ and $\exp$ in (\ref{eq:df4exp}) is the exponential map between the Lie algebra $\fd_3$ and its Lie group $\fD_3$ with inverse $\log$ in (\ref{eq:df4log}).
\exmsymbol
\end{exm}

The previous example of the Lie algebra $\fd_3$ of the Lie group $\fD_3$ holds for all $n\in\nset$ and $d\in\nset_0$.
We define the following.

\begin{dfn}\label{dfn:dd}
Let $n\in\nset$ and $d\in\nset_0$.
We define
\[\fd_d := \left\{\sum_{\alpha\in\nset_0^n\setminus\{0\}: |\alpha|\leq d} d_\alpha\cdot\partial^\alpha \,\middle|\, d_\alpha\in\rset\ \text{for all}\ \alpha\in\nset_0^n\setminus\{0\}\ \text{with}\ |\alpha|\leq d\right\}.\]
\end{dfn}

It is clear that $(\fd_d,\,\cdot\,,+)$ is an algebra on $\rset[x_1,\dots,x_n]_{\leq d}$ and we have the following.

\begin{thm}\label{thm:liealgebrad}
Let $n\in\nset$ and $d\in\nset_0$.
Then $(\fd_d,\,\cdot\,,+)$ is the Lie algebra of the Lie group $(\fD_d,\,\cdot\,)$ with exponential map
\[\exp: \fd_d\to\fD_d,\quad A\mapsto \sum_{k=0}^d \frac{A^k}{k!}\]
with inverse
\[\log:\fD_d\to\fd_d,\quad A\mapsto -\sum_{k=1}^d \frac{(\one-A)^k}{k}.\]
\end{thm}
\begin{proof}
Follows from \Cref{lem:expLog} similar to \Cref{exm:n1d3cont}.
\end{proof}

\subsection{The Regular Fr\'echet Lie Group $\fD$ and its Lie Algebra $\fd$}
\label{sec:frechetLieGroupD}

In \Cref{sec:liegroup} and \ref{sec:liealgebra} we have seen that $(\fD_d,\,\cdot\,)$ is a Lie group with Lie algebra $(\fd_d,\,\cdot\,,+)$ for all $d\in\nset_0$.
Hence, similar to \Cref{dfn:dd} we define the following.

\begin{dfn}
Let $n\in\nset$. We define
\[\fd := \left\{ \sum_{\alpha\in\nset_0^n\setminus\{0\}} d_\alpha\cdot\partial^\alpha \,\middle|\, d_\alpha\in\rset\ \text{for all}\ \alpha\in\nset_0^n\setminus\{0\}\right\}.\]
\end{dfn}

For $\fD$ and $\fd$ we have the following.

\begin{thm}\label{thm:liealgebra}
Let $n\in\nset$.
Then the following hold:
\begin{enumerate}[(i)]
\item $(\fd,\,\cdot\,,+)$ is a commutative algebra.

\item $(\fD,\,\cdot\,)$ is a commutative group.

\item The map
\[\exp:\fd\to\fD,\quad A\mapsto \sum_{k\in\nset_0} \frac{A^k}{k!}\]
is bijective.

\item The map
\[\log:\fD\to\fd,\quad A\mapsto -\sum_{k\in\nset} \frac{(\one-A)^k}{k}\]
is bijective.

\item The maps $\exp:\fd\to\fD$ and $\log:\fD\to\fd$ are inverse to each other.
\end{enumerate}
\end{thm}
\begin{proof}
(i): That is clear.

(ii): That $A\cdot B = B\cdot A$ holds for all $A,B\in\fD$ is clear.
The inverse of $A\in\fD$ is uniquely determined by solving (\ref{eq:inverseSystem}) to get (\ref{eq:inverseA}) for all $\gamma\in\nset_0^n$.
This is a formal power series argument with coordinate-wise convergence (as in the Fr\'echet topology, \Cref{exm:powerSeriesTop}).

(iii): At first we show that $\exp:\fd\to\fD$ is well-defined.
To see this note that for any $A\in\fd$ we have
\[A^k = \sum_{\alpha\in\nset_0^n: |\alpha|\geq k} c_\alpha\cdot\partial^\alpha,\]
i.e., $A^k$ contains no differential operators of order $\leq k-1$.
Hence, the sum
\[\sum_{k=0}^K \frac{A^k}{k!} = \sum_{\alpha\in\nset_0^n} c_{K,\alpha}\cdot\partial^\alpha \]
converges coefficient-wise to
\[\exp A =\sum_{k\in\nset_0} \frac{A^k}{k!} = \sum_{\alpha\in\nset_0^n} c_\alpha\cdot\partial^\alpha,\]
i.e., in the Fr\'echet topology of $\fD\subsetneq\rset[[\partial_1,\dots,\partial_n]] \cong \rset[[x_1,\dots,x_n]]$, see \Cref{exm:powerSeriesTop}.
In other words, the coefficients $c_\alpha$ depend only on $A^k$ for $k=0,\dots,|\alpha|$, we therefore have $c_{K,\alpha} = c_\alpha$ for all $K>|\alpha|$, and hence $\exp A\in\fD$ is well-defined.
With that we have $\exp\fd\subseteq\fD$.
For equality we give the inverse map in (v).

(iv): To show that $\log:\fD\to\fd$ is well-defined the same argument as in (iii) holds for $(\one-A)^k$ with $A\in\fD$.
It shows that $\log A\in\fd$ for all $A\in\fD$ is well-defined and we have $\log\fD\subseteq\fd$.

(v): To prove that $\exp$ and $\log$ are inverse to each other we remark
\[\rset[x_1,\dots,x_n] = \bigcup_{d\in\nset_0} \rset[x_1,\dots,x_n]_{\leq d}.\]
Define
\[\exp_d A := \sum_{k=0}^d \frac{A^k}{k!} \quad\text{and}\quad \log_d A := -\sum_{k=1}^d \frac{(\one-A)^k}{k}.\]
Then for every $p\in\rset[x_1,\dots,x_n]$ with $d = \deg p$ we have
\[\exp(\log A)p = \exp_d(\log_d A)p = Ap\]
for all $A\in\fd$ by \Cref{thm:liealgebrad}, i.e., $\exp(\log A) = A$ for all $A\in\fD$.
Similarly, we have $\log(\exp A)p = Ap$ for all $A\in\fd$.
This also shows the remaining assertions $\exp(\fd) = \fD$ and $\log \fD = \fd$ from (iii) and (iv).
\end{proof}

For $\rset[[\partial_1,\dots,\partial_n]]\supseteq\fd,\fD$ being a Fr\'echet space with the before mentioned topology (of coordinate-wise convergence, \Cref{exm:powerSeriesTop}) we have the following.

\begin{cor}\label{cor:lieprops}
Let $n\in\nset$ and let $\fd,\fD\subseteq\rset[[\partial_1,\dots,\partial_n]]$ be Fr\'echet spaces (equipped with the coordinate-wise convergence).
Then the following hold:
\begin{enumerate}[(i)]
\item $\fD\times\fD\to\fD$, $(A,B)\mapsto AB^{-1}$ is smooth.

\item $\exp:\fd\to\fD$ is smooth and
\[\left.\frac{\diff}{\diff t}\right|_{t=0} \exp(t u) = u\]
holds for all $u\in\fg$.

\item $\log:\fd\to\fD$ is smooth.
\end{enumerate}
\end{cor}
\begin{proof}
(i): Let $A = \sum_{\alpha\in\nset_0^n} a_\alpha\partial^\alpha$ and $B = \sum_{\alpha\in\nset_0^n} b_\alpha\partial^\alpha$ with $a_0 = b_0 = 1$. From (\ref{eq:inverseSystem}) we see that the multiplication is smooth since every coordinate $c_\gamma$ of the product $AB = \sum_{\alpha\in\nset_0^n} c_\alpha\partial^\alpha$ is a polynomial in $a_\alpha$ and $b_\alpha$ with $|\alpha|\leq |\gamma|$.
The inverse $B^{-1} = \sum_{\alpha\in\nset_0^n} d_\alpha\partial^\alpha$ is smooth because of (\ref{eq:inverseA}), i.e., also the coefficients $d_\gamma$ of the inverse depend polynomially on the coefficients $b_\alpha$ with $|\alpha|\leq |\gamma|$.

(ii): In the proof of \Cref{thm:liealgebra} (iii) we have already seen that the coefficients $c_\gamma$ of $\exp A = \sum_{\alpha\in\nset_0^n} c_\alpha\partial^\alpha$ depend polynomially on the coefficients $a_\alpha$ of $A = \sum_{\alpha\in\nset_0^n\setminus\{0\}} a_\alpha\partial^\alpha$ with $|\alpha|\leq |\gamma|$. 

The condition
\[\left.\frac{\diff}{\diff t}\right|_{t=0} \exp(t u) = u\]
then follows by direct calculations.

(iii): Follows like (ii) from \Cref{thm:liealgebra} (iv).
\end{proof}

It is easy to see that $\fd$ and $\fD$ are both infinite dimensional smooth (Fr\'echet) manifolds.
Hence, summing everything up we have the following.

\begin{thm}\label{thm:mainFrechetLieGroups}
Let $n\in\nset$.
Then $(\fD,\,\cdot\,)$ as a Fr\'echet space is a commutative regular Fr\'echet Lie group with the commutative Fr\'echet Lie algebra $(\fd,\,\cdot\,,+)$.
The exponential map
\[\exp:\fd\to\fD,\quad A\mapsto \sum_{k\in\nset_0} \frac{A^k}{k!}\]
is smooth and bijective with the smooth and bijective inverse
\[\log:\fD\to\fd,\quad A\mapsto -\sum_{k\in\nset} \frac{(\one-A)^k}{k}.\]
\end{thm}
\begin{proof}
We have that $\fD$ is an infinite dimensional smooth manifold, $\fD$ is a Fr\'echet space (with the coefficient-wise convergence topology, \Cref{exm:powerSeriesTop}) and by \Cref{thm:liealgebra} (ii) we also have that $(\fD,\,\cdot\,)$ is a commutative group.
By \Cref{cor:lieprops} (i) we have that $(A,B)\mapsto AB^{-1}$ is continuous in the Fr\'echet topology.
Hence, $(\fD,\,\cdot\,)$ is an infinite dimensional commutative Fr\'echet Lie group.

The properties about $\exp$ and $\log$ are \Cref{thm:liealgebra} (iii)--(v).

We now prove the regularity condition (vi) in \Cref{dfn:frechetLieGroup}.
Let $F:\rset\to\fD$ be a $C^1$-differentiable function, i.e.,
\begin{equation}\label{eq:diffDef}
\lim_{n\to\infty} \left( F\!\left(t+\frac{s}{n}\right)\cdot F(t)^{-1} \right)
\end{equation}
converges uniformly on each compact interval to a one-parameter subgroup
\[\exp (s f(t))\]
where $f:\rset\to\fd$ is the derivative $\dot{F}(t)$ of $F(t)$ at $t\in\rset$, see \cite[p.\ 10]{omori97}.
But we can take the logarithm of $F$
\[\tilde{f}(t) := \log F(t)\]
for all $t\in\rset$ to see that $\tilde{f}$ is $C^1$ since $\log$ is smooth by \Cref{cor:lieprops} (iii).
By \Cref{cor:lieprops} (ii) we have $f = \tilde{f}$.
Hence, $C^1(\fD,\fd)$ coincides with the set of all $C^1$-curves in $\fD$ under the Fr\'echet topology of $\fD$.
\end{proof}

In the previous proof we can also replace (\ref{eq:diffDef}) by the fact that a function $F:\rset\to\rset[[\partial_1,\dots,\partial_n]]$, $t\mapsto F(t) = \sum_{\alpha\in\nset_0^n} F_\alpha(t)\cdot\partial^\alpha$ is $C^m$ for some $m\in\nset_0$ if and only if every coordinate function $F_\alpha$ is $C^m$.

\subsection{L\'evy--Khinchin Formula}

A measure $\mu$ on $\rset^n$ is called \emph{divisible by $k\in\nset$} if there exists a measure $\nu$ such that $\mu = \nu^{*k}$.
A measure $\mu$ on $\rset^n$ is called \emph{infinitely divisible} if it is divisible by any $k\in\nset$.
Infinitely divisible measures are fully characterized by the \emph{L\'evy--Khinchin formula}.

\begin{thm}[L\'evy--Khinchin, see e.g.\ {\cite[Cor.\ 15.8]{kallenberg02}} or {\cite[Satz 16.17]{klenkewtheorie}}]\label{thm:leviKhinchinFormula}
Let $n\in\nset$ and $\mu$ be a measure on $\rset^n$.
The following are equivalent:
\begin{enumerate}[(i)]
\item $\mu$ is infinitely divisible.

\item There exist a vector $b\in\rset^n$, a symmetric matrix $\Sigma\in\rset^{n\times n}$ with $\Sigma\succeq 0$, and a measure $\nu$ on $\rset^n$ such that
\[\log \int e^{itx}~\diff\mu(x) = itb - \frac{1}{2} t^T\Sigma t + \int (e^{itx} - 1 - itx \cdot\chi_{\|x\|_2<1})~\diff\mu(x)\]
holds for the characteristic function of $\mu$.
\end{enumerate}
\end{thm}

\section{The Regular Fr\'echet Lie Group Structure of Sequences}
\label{sec:sequences}

We define appropriate sets $\fS,\fs\subsetneq\rset^{\nset_0^n}$ of sequences with convolution $*$.
From these definitions we will see that the set of sequences $\rset^{\nset_0^n}$, and therefore also $\fS$ and $\fs$, inherit the Fr\'echet topology (\Cref{exm:powerSeriesTop}) from $\rset[[x_1,\dots,x_n]]$.

\begin{dfn}\label{dfn:sSD}
Let $n\in\nset$.
We define
\[\fS := \left\{ (s_\alpha)_{\alpha\in\nset_0^n}\in\rset^{\nset_0^n} \,\middle|\, s_0 = 1\right\}
\quad\text{and}\quad
\fs := \left\{ (s_\alpha)_{\alpha\in\nset_0^n}\in\rset^{\nset_0^n} \,\middle|\, s_0 = 0\right\}.\]
Define the linear and bijective map
\begin{equation}\label{eq:Dmap}
D:\rset^{\nset_0^n}\to\rset[[\partial_1,\dots,\partial_n]],\quad s=(s_\alpha)_{\alpha\in\nset_0^n}\mapsto D(s) := \sum_{\alpha\in\nset_0^n} \alpha!\cdot s_\alpha\cdot\partial^\alpha
\end{equation}
and on $\rset^{\nset_0^n}$ the convolution $*$ as
\[*:\rset^{\nset_0^n}\times\rset^{\nset_0^n}\to\rset^{\nset_0^n},\quad (s,t)\mapsto s*t := D^{-1}(D(s)D(t)).\]
We abbreviate
\[s^0 = \one := (\delta_{0,\alpha})_{\alpha\in\nset_0^n} \qquad\text{and}\qquad s^{*k} := \underbrace{s* \cdots * s}_{k\text{-times}}\]
for all $k\in\nset$ and $s\in\rset^{\nset_0^n}$.
\end{dfn}

The map $D$ in (\ref{eq:Dmap}) is of course a Fr\'echet space isomorphism. 
The map $D$ also equips $\fS$ and $\fs$ with the Fr\'echet topology of $\fD$ and $\fd$, respectively.
In summary, we have the following.

\begin{cor}\label{cor:groupIso}
Let $n\in\nset$. Then the following hold:
\begin{enumerate}[(i)]
\item $s*t\in\fS$ holds for all $s,t\in\fS$.
\item $D|_\fS:(\fS,*)\to(\fD,\,\cdot\,)$ is a Fr\'echet group isomorphisms.
\item $s*t\in\fs$ and $s+t\in\fs$ hold for all $s,t\in\fs$.
\item $D|_\fs:(\fs,*,+)\to(\fd,\,\cdot\,,+)$ is a Fr\'echet algebra isomorphisms.
\end{enumerate}
\end{cor}

\begin{thm}\label{thm:sequenceLieGroupAlgebra}
Let $n\in\nset$.
Then $(\fS,*)$ is a commutative regular Fr\'echet Lie group with the commutative Lie algebra $(\fs,*,+)$.
The exponential map
\[\exp:\fs\to\fS,\quad s\mapsto \exp(s) := \sum_{k\in\nset_0} \frac{s^{*k}}{k!}\]
is smooth and bijective with the smooth and bijective inverse
\[\log:\fS\to\fs,\quad s\mapsto \log(s) := -\sum_{k\in\nset} \frac{(\one-s)^{*k}}{k}.\]
\end{thm}
\begin{proof}
From \Cref{cor:groupIso} we see that the map $D:\rset^{\nset_0^n}\to\rset^{\nset_0^n}$ induces the Fr\'echet Lie group isomorphism $D|_\fS:(\fS,*)\to(\fD,\,\cdot\,)$ and the Fr\'echet Lie algebra isomorphism $D|_\fs:(\fs,*,+)\to(\fd,\,\cdot\,,+)$.
Apply both isomorphisms to \Cref{thm:mainFrechetLieGroups} and the assertion is proved.
\end{proof}

We want to point out that Hirschman and Widder \cite{hirsch55} extensively investigated the inversion theory of convolution of measures and functions.
But from \Cref{thm:sequenceLieGroupAlgebra} we see that the convolution of their moments  is trivial when we allow signed moment sequences.

For $*$ on $\fS$ we find from (\ref{eq:measureConv}) the following.

\begin{cor}\label{cor:signedConv}
Let $n\in\nset$.
If $s\in\fS$ is represented by the signed measure $\mu$ and $t\in\fS$ is represented by the signed representing measure $\nu$ then $s*t$ is represented by the measure $\mu*\nu$.
\end{cor}

Note, that every sequence $s\in\rset^{\nset_0^n}$ is represented by a signed measure, see e.g.\ \cite{polya38,boas39a,sherman64,schmud78}.
It is even possible to restrict the support to $[0,\infty)^n$, see e.g.\ \cite{boas39a,sherman64} or to use only linear combinations of Dirac measures \cite{bloom53}.

Because of the connection of positivity preservers to moment sequences by \Cref{thm:posPresChar} we define the following.

\begin{dfn}
Let $n\in\nset$.
We define
\[\fS_+ := \{s\in\fS \,|\, s\ \text{is a moment sequence}\}.\]
\end{dfn}

Then $\fS_+$ is a base of the moment cone.

\begin{rem}
For a general moment sequence $s = (s_\alpha)_{\alpha\in\nset_0^n}$ we only have $s_0 > 0$.
But, scaling $\tilde{s}= s_0^{-1}\cdot s\in\fS$ gives $s = \exp(\ln s_0 + \log \tilde{s})$.
\exmsymbol
\end{rem}

The Fr\'echet topology on $\fS$ and \Cref{cor:signedConv} imply the following.

\begin{cor}\label{cor:fSplusProperties}
Let $n\in\nset$.
The following hold:
\begin{enumerate}[(i)]
\item The set $\fS_+$ is convex and closed.

\item For all $s,t\in\fS_+$ we have $s*t\in\fS_+$.
\end{enumerate}
\end{cor}
\begin{proof}
(i): Convexity is clear.
It suffices to prove that $\fS_+$ is closed in the Fr\'echet topology.

Let $s^{(n)} = (s_\alpha^{(n)})_{\alpha\in\nset_0^n}\in\fS_+$ for all $n\in\nset_0$ be such that $s^{(n)}\to s$ in the Fr\'echet topology (\Cref{exm:powerSeriesTop}), i.e., $s_\alpha^{(n)}\to s_\alpha$ for all $\alpha\in\nset_0^n$.
Define $L_s:\rset[x_1,\dots,x_n]\to\rset$ by $L_s(x^\alpha) := s_\alpha$ for all $\alpha\in\nset_0^n$ and extend it linearly to $\rset[x_1,\dots,x_n]$.
Let $p\in\pos(\rset^n)$ then $0\leq L_{s^{(n)}}(p)\to L_s(p)\geq 0$ shows that $L_s$ is $\pos(\rset^n)$-positive, i.e., by Haviland's Theorem \cite{havila36} $L_s$ is a moment functional and $s\in\cS_+$.

(ii): Since $s,t\in\fS_+$ we have that $s$ is represented by $\mu$ and $t$ is represented by $\nu$.
Therefore, $s*t$ is represented by $\mu*\nu$ by \Cref{cor:signedConv}.
\end{proof}

The following is another trivial consequence of (\ref{eq:measureConv}).

\begin{cor}
Let $n\in\nset$ and $s,t\in\fS_+$.
If $s$ or $t$ is indeterminate then $s*t\in\fS_+$ is indeterminate.
\end{cor}

\section{Generators of Positivity Preservers with Constant Coefficients}
\label{sec:generators}

In \Cref{sec:frechetLieGroupD} we proved that $\fD$ is a Fr\'echet Lie group with Fr\'echet Lie algebra $\fd$.
With the smooth and bijective $\exp:\fd\to\fD$ with the inverse $\log:\fD\to\fd$ we can now easily go from positivity preservers with constant coefficients to their generators.

\begin{dfn}\label{dfn:fd+}
Let $n\in\nset$. We define the set
\[\fD_+ := \{A\in\fD \,|\, A\ \text{is a positivity preserver}\}\]
of all \emph{positivity preservers with constant coefficients} and we define the set
\[\fd_+:=\left\{A\in\fd\,\middle|\,\exp(tA)\in\fD_+\ \text{for all}\ t\geq 0\right\}\]
of all \emph{generators of positivity preservers with constant coefficients}.
\end{dfn}

From \Cref{thm:posPresChar} we see that $D|_{\fS_+}:\fS_+\to\fD_+$ in (\ref{eq:Dmap}) is an isomorphism.

\begin{exms}\label{exm:fdPlus}
Let $n\in\nset$.
Then we have the following:
\begin{enumerate}[\bfseries\ (a)]
\item $c\cdot\Delta = c\cdot(\partial_1^2 + \dots + \partial_n^2)\in\fd_+$ for all $c\geq 0$ since $\Delta$ generates the heat equation/kernel.

\item $c_1\cdot\partial_1+ \dots + c_n\cdot\partial_n\in\fd_+$ for all $c_1,\dots,c_n\in\rset$ since $\partial_i$ generates the translation group in the direction $(c_1,\dots,c_n)\in\rset^n$.\exmsymbol
\end{enumerate}
\end{exms}

\begin{exm}\label{exm:notpos2}
Let $k\in\nset$ with $k\geq 3$.
Then $\partial_i^k\not\in\fd_+$, see \Cref{exm:notpos}.\exmsymbol
\end{exm}

The following result shows that the cases in \Cref{exm:fdPlus} are the only generators of positivity preserves of finite rank.

\begin{lem}
Let $A = \sum_{j=1}^k a_j\partial_x^j\in\fd_+$. Then $k\leq 2$.
\end{lem}
\begin{proof}
Let $k\geq 3$ and $a_k = 1$.
By \Cref{exm:notpos} (and \ref{exm:notpos2}) we have that $\exp(\partial_x^k)\not\in\fD_+$, i.e., it is not a positivity preserver.
Hence, there exists a $f_0\in\rset[x]$ with $f_0\geq 0$ and $x_0\in\rset$ such that $[\exp(\partial_x^k) f_0](x_0) = -1$.

Assume to the contrary that $A\in\fd_+$. By scaling $x$ and $A$ we have that
\begin{equation}\label{eq:aScaling}
A_\lambda := \sum_{j=1}^k \lambda^{k-j} a_j \partial_x^j\quad\in\fd_+
\end{equation}
holds for all $\lambda > 0$.
By \Cref{thm:mainFrechetLieGroups} we have that $[\exp(A_\lambda) f_0](x_0)$ is continuous in $\lambda$.
Since $(\exp(A_0) f_0)(x_0)=-1$ there exists a $\lambda_0 > 0$ such that $[\exp(A_{\lambda_0}) f_0](x_0) < 0$, i.e., $A\not\in\fd_+$ and therefore we have $k\leq 2$.
\end{proof}

It is easy to see that the previous result also holds for $n\geq 2$.
To see this let $\alpha\in\nset_0^n$ with $|\alpha|\geq 3$.
Then from \Cref{thm:posPresChar} it follows that $\exp(\partial^\alpha)\not\in\fD_+$.
Choosing the same scaling argument (\ref{eq:aScaling}) we find $\partial^\alpha\not\in\fd_+$.

For $\fD_+$ and $\fd_+$ the following holds.

\begin{cor}\label{cor:fDplusfdplusProps}
Let $n\in\nset$.
Then the following hold:
\begin{enumerate}[(i)]
\item $\fD_+$ is a closed and convex set.

\item $\fd_+$ is a non-trivial, closed, and convex cone.
\end{enumerate}
\end{cor}
\begin{proof}
(i): Follows from \Cref{cor:fSplusProperties} (i) with the bijective and linear map $D$ from \Cref{dfn:sSD} eq.\ (\ref{eq:Dmap}).

(ii):
For the non-triviality we have the non-trivial \Cref{exm:fdPlus}.

For the convexity let $A,B\in\fd_+$.
Since $(\fd,\,\cdot\,,+)$ is a commutative algebra we have that $\exp(t(A+B)) = \exp(tA) \exp(tB)$ and since the product of two positivity preservers is again a positivity preserver we have $A+B\in\fd_+$.

For the closeness let $A_n\in\fd_+$ for all $n\in\nset_0$ with $A_n\to A$ in the Fr\'echet topology of $\rset[[\partial_1,\dots,\partial_n]]$, see \Cref{exm:powerSeriesTop}.
By \Cref{thm:sequenceLieGroupAlgebra} we have that $\exp:\fd\to\fD$ is smooth, i.e., especially continuous.
Hence,
\[\fD_+\ni\exp(t A_n)\;\to\;\exp(t A)\in\fD_+\]
for all $t\geq 0$ since $\fD_+$ is closed by (i).
Hence, $A\in\fd_+$.

For the cone property let $A\in\fd_+$ then also $cA\in\fd_+$ for all $c\geq 0$.
\end{proof}

\begin{cor}\label{cor:posPDE}
Let $n\in\nset$ and $A\in\fd$.
Then the following are equivalent:
\begin{enumerate}[(i)]
\item $A\in\fd_+$.
\item The unique solution $p_t$ of
\begin{equation}\label{eq:timePDE}
\partial_t p = Ap
\end{equation}
for any initial value $p_0\in\pos(\rset^n)$ fulfills $p_t\in\pos(\rset^n)$ for all $t\geq 0$.
\end{enumerate}
\end{cor}
\begin{proof}
Since $p_t = \exp(tA)p_0$ is the unique solution of the time evolution (\ref{eq:timePDE}) we have that (i) $\Leftrightarrow$ $\exp(tA)$ is a positivity preserver for all $t\geq 0$ $\Leftrightarrow$ (ii).
\end{proof}

While we have $\fd_+ \subseteq \log\fD_+$ equality does not hold as we will see in \Cref{cor:generatorsNotLog}.
The existence of a positivity preserver is equivalent to the existence of an infinitely divisible representing measure as the following result shows.

\begin{mthm}\label{thm:infinitelyDivisibleMeasure}
Let $n\in\nset$.
Then the following are equivalent:
\begin{enumerate}[(i)]
\item $A\in\fd_+$.

\item $\exp A$ has an infinitely divisible representing measure.

\item $\exp(tA)$ has an infinitely divisible representing measure for some $t>0$.

\item $\exp(tA)$ has an infinitely divisible representing measure for all $t>0$.
\end{enumerate}
\end{mthm}
\begin{proof}
(i) $\Rightarrow$ (ii):
Let $A\in\fd_+$, i.e., $\exp (t A)\in\fD_+$ has a representing measure $\mu_t$ for all $t\in [0,\infty)$.
Set $\nu_k := (\mu_{1/k!})^{*k!}$.
Then $\nu_k$ is a representing measure of $\exp A $ for all $k\in\nset$.
Since $\rset[x_1,\dots,x_n]$ is an adapted space $(\nu_k)_{k\in\nset}$ is vaguely compact by \cite[Thm.\ 1.19]{schmudMomentBook} and there exists a subsequence $(k_i)_{i\in\nset}$ such that $\nu_{k_i}\to\nu$ and $\nu$ is a representing measure of $\exp A$.

It remains to show that $\nu$ is infinitely divisible, i.e., for every $l\in\nset$ there exists a measure $\omega_l$ with $\omega_l^{*l} = \nu$.

Let $l\in\nset$.
For $i\geq l$ we define
\[\omega_{l,i}:= (\mu_{1/k_i!})^{*k_i!/l}\]
i.e., $\omega_{l,i}$ is a representing measure of $\exp(A/l)$.
Again, since $\rset[x_1,\dots,x_n]$ is an adapted space by \cite[Thm.\ 1.19]{schmudMomentBook} there exists a subsequence $(i_j)_{j\in\nset}$ such that $\omega_{l,i_j}$ converges to some $\omega_l$, i.e., $\omega_{l,i_j}\xrightarrow{j\to\infty}\omega_l$.
Hence,
\[(\omega_l)^{*l} = \lim_{j\to\infty} (\omega_{l,i_j})^{*l} = \lim_{j\to\infty} \nu_{k_{i_j}} = \nu,\]
i.e., $\nu$ is divisible by all $l\in\nset$ and hence $\nu$ is an infinitely divisible representing measure of $\exp A$.

(ii) $\Rightarrow$ (i):
Let $\mu_1$ be an infinitely divisible representing measure of $\exp A$.
Then
\[\mu_q := \mu_1^{*q}\]
exists for all $q\in\qset\cap[0,\infty)$ and it is a representing measure of $\exp(qA)$, i.e., $\exp(qA)\in\fD_+$ for all $q\in [0,\infty)\cap\qset$.
Since by \Cref{thm:mainFrechetLieGroups} $\exp:\fd\to\fD$ is continuous and by \Cref{cor:fDplusfdplusProps} (i) $\fD_+$ is closed we have that $\exp(qA)\in\fD_+$ for all $q\geq 0$.
Hence, we have $A\in\fd_+$.

(iv) $\Rightarrow$ (iii): Clear.

(iii) $\Rightarrow$ (i): By ``(ii) $\Leftrightarrow$ (i)'' we have that $qtA\in\fd_+$ for $t>0$ and all $q\in [0,\infty)\cap\qset$.
Since $\fd_+$ is closed by \Cref{cor:fDplusfdplusProps} (ii) we have $q_itA\to A\in\fd_+$ for $q_i\in\qset$ with $q_i\to t^{-1}$ as $i\to\infty$.

(i) $\Rightarrow$ (iv): Since $A\in\fd_+$ and $\fd_+$ is a closed convex cone by \Cref{cor:fDplusfdplusProps} (ii) we have that $tA\in\fd_+$ for all $t>0$ and hence by ``(i) $\Leftrightarrow$ (ii)'' we have that $\exp(tA)$ has an infinitely divisible representing measure for all $t>0$.
\end{proof}

\begin{exm}[Example \ref{exm:fdPlus} (b) continued]
Let $n=1$ and $a\in\rset$.
Then
\[\exp(r\partial_x) = \sum_{k\in\nset_0} \frac{a^k}{k!}\cdot\partial_x^k\]
is represented by $\mu = \delta_a$ since $\delta_a$ is the representing measure of the moment sequence $(a^k)_{k\in\nset_0}$.
For any $r>0$ we have $\delta_{a/r}^{*r} = \delta_a$, i.e., $\delta_a$ is infinitely divisible.
In fact, $\delta_a$ are the only compactly supported infinite divisible measures, see e.g.\ \cite[p.\ 316]{klenkewtheorie}.
Hence, by \Cref{thm:infinitelyDivisibleMeasure} we have $\partial_x\in\fd_+$.
\exmsymbol
\end{exm}

\begin{exm}\label{exm:nonPosGen}
Let $A\in\fD_+$ be the positivity preserver represented by the measure
\[\diff\mu = \chi_{[0,1]^n}~\diff\lambda\]
where $\lambda$ is the $n$-dimensional Lebesgue measure and $\chi_{[0,1]^n}$ is the characteristic function of $[0,1]^n$.
Since $\supp\mu$ is compact $\mu$ is unique.

It is known that the only infinitely divisible measures with compact support are $\delta_x$ for $x\in\rset^n$, see e.g.\ \cite[p.\ 316]{klenkewtheorie}.
Therefore, we have that $\mu$ is not infinitely divisible and hence $\log A\not\in\fd_+$.
\exmsymbol
\end{exm}

The previous example implies that the inclusion $\fd_+\subseteq\log\fD_+$ is proper.

\begin{cor}\label{cor:generatorsNotLog}
Let $n\in\nset$.
Then $\fd_+ \subsetneq \log\fD_+$.
\end{cor}
\begin{proof}
We have $\log\fD_+\setminus\fd_+\neq\emptyset$ by \Cref{exm:nonPosGen}.
\end{proof}

We have seen in \Cref{thm:infinitelyDivisibleMeasure} the one-to-one correspondence between a positivity preserver $A\in\fD_+$ having an infinitely divisible representing measure and $A\in\fD_+$ having a generator.
The infinitely divisible measures are fully characterized by the L\'evy--Khinchin formula, see \Cref{thm:leviKhinchinFormula}.
The L\'evi--Khinchin formula is used in the following result to fully characterize the generators $\fd_+$ of the positivity preservers $\fD_+$.

\begin{mthm}\label{thm:mainPosGenerators}
Let $n\in\nset$.
The following are equivalent:
\begin{enumerate}[(i)]
\item $A = \sum_{\alpha\in\nset_0^n\setminus\{0\}} \frac{a_\alpha}{\alpha!}\cdot\partial^\alpha \in\fd_+$.

\item There exists a symmetric matrix $\Sigma = (\sigma_{i,j})_{i,j=1}^n\in\rset^n$ with $\Sigma\succeq 0$, a vector $b = (b_1,\dots,b_n)^T\in\rset^n$, and a measure $\nu$ on $\rset^n$ with
\[\int_{\|x\|_2\geq 1} |x_i|~\diff\nu(x)<\infty \qquad\text{and}\qquad \int_{\rset^n} |x^\alpha|~\diff\nu(x)<\infty\]
for all $i=1,\dots,n$ and $\alpha\in\nset_0^n$ with $|\alpha|\geq 2$ such that
\begin{align*}
a_{e_i} &= b_i + \int_{\|x\|_2\geq 1} x_i~\diff\nu(x) && \text{for all}\ i=1,\dots,n,\\
a_{e_i+e_j} &= \sigma_{i,j} + \int_{\rset^n} x^{e_i + e_j}~\diff\nu(x) && \text{for all}\ i,j=1,\dots,n,
\intertext{and}
a_\alpha &= \int_{\rset^n} x^\alpha~\diff\nu(x) &&\text{for all}\ \alpha\in\nset_0^n\ \text{with}\ |\alpha|\geq 3.
\end{align*}
\end{enumerate}
\end{mthm}
\begin{proof}
By \Cref{thm:infinitelyDivisibleMeasure} ``(i) $\Leftrightarrow$ (ii)'' we have that (i) $A\in\fd_+$ if and only if $\exp A$ has an infinitely divisible representing measure $\mu$, i.e., by \Cref{thm:posPresChar} we have
\begin{equation}\label{eq:levi1}
\exp A = \sum_{\alpha\in\nset_0^n} \frac{1}{\alpha!}\cdot \int_{\rset^n} x^\alpha~\diff\mu(x)\cdot\partial^\alpha.
\end{equation}
By \Cref{thm:mainFrechetLieGroups} we can take the logarithm and hence (\ref{eq:levi1}) is equivalent to
\begin{equation}\label{eq:levi2}
A = \sum_{\alpha\in\nset_0^n\setminus\{0\}} \frac{a_\alpha}{\alpha!}\cdot\partial^\alpha = \log \left(\sum_{\alpha\in\nset_0^n} \frac{1}{\alpha!}\cdot \int_{\rset^n} x^\alpha~\diff\mu(x)\cdot\partial^\alpha\right).
\end{equation}
With the isomorphism
\[\cset[[\partial_1,\dots,\partial_n]]\to\cset[[t_1,\dots,t_n]],\ \partial_1\mapsto i t_1,\ \dots,\ \partial_n\mapsto it_n\]
we have that (\ref{eq:levi2}) is equivalent to
\begin{align}
\sum_{\alpha\in\nset_0^n\setminus\{0\}} \frac{a_\alpha}{\alpha!}\cdot(it)^\alpha
&= \log \left(\sum_{\alpha\in\nset_0^n} \frac{1}{\alpha!}\cdot \int_{\rset^n} x^\alpha~\diff\mu(x)\cdot (it)^\alpha\right).\label{eq:levi3}
\intertext{But the right hand side of (\ref{eq:levi3}) is now the characteristic function}
&= \log \int e^{i tx}~\diff\mu(x)\notag
\intertext{of the $\mu$.
Hence, by the L\'evy--Khinchin formula (see \Cref{thm:leviKhinchinFormula}) we have}
&= i bt -\frac{1}{2}t^T\Sigma t + \int (e^{itx} - 1 - itx \cdot\chi_{\|x\|_2<1})~\diff\nu(x).\label{eq:levi4}
\end{align}
After a formal power series expansion of $e^{itx}$ in the Fr\'echet topology of $\cset[[x_1,\dots,x_n]]$, see \Cref{exm:powerSeriesTop}, and a comparison of coefficients we have that (\ref{eq:levi4}) is equivalent to (ii) which ends the proof.
\end{proof}

Form the previous result we see that the difference between $\fd_+$ and a moment sequence is that the representing (L\'evy) measure $\nu$ in (\ref{eq:levi4}) can have a singularity of order $\leq 2$ at the origin.

\section{A Strange Action on $\pos([0,\infty))$}
\label{sec:strange}

Before we end this work we want to discuss an example of a strange positivity action on $[0,\infty)$.
We take the following example.

\begin{exm}
Let $\Delta=\partial_x^2$ on $L^2([0,\infty),\rset)$ with Dirichlet boundary conditions
\begin{align*}
\partial_t u &= \Delta u\\
u(0,t) &= 0 \quad\text{for all}\ t\geq 0\\
u(\,\cdot\,,0) &= u_0\in L^2([0,\infty),\rset).
\end{align*}
Then
\[u(x,t) = \int_0^\infty k_t(x,y) u_0(y)~\diff y\]
with
\[k_t(x,y) = \frac{1}{\sqrt{4\pi t}}\cdot e^{-(x-y)^2/4}\cdot \left[1-e^{xy/t}\right]\]
for all $x,y\geq 0$ and $t>0$, see \cite[Exm.\ 4.1.1]{davis89}.

For $u_0\in\cS([0,\infty),\rset)$ a Schwartz function on $[0,\infty)$ we have that $k_t*u_0\in\cS([0,\infty),\rset)$ for all $t>0$ and hence we look at the time-dependent moments
\[s_{j}(t) = \int_K x^j\cdot u(x,t)~\diff x = \int_{[0,\infty)^2} x^j\cdot k_t(x,y)\cdot u_0(y)~\diff x~\diff y\]
for all $j\in\nset_0$ and $t>0$.

The action on $\rset[x]$ is then given by $(T_tp)(x) = \int_0^\infty p(y)\cdot k_t(x,y)~\diff y$. We find for $p=1$ that
\[(T_t1)(x) = \int_0^\infty k_t(x,y)~\diff y\]
is a continuous, non-decreasing function in $x\in [0,\infty)$ with $(T_t1)(0) = 0$ and $\lim_{x\to\infty} (T_t1)(x) = 1$, i.e., $T_t1\notin\rset[x]$ for any $t\in (0,\infty)$.
While for $p=x^j$ with $j\in\nset$ we get for the time-dependent moments by partial integration
\begin{align*}
\partial_t s_j(t) &= \partial_t \int_0^\infty x^j\cdot u(x,t)~\diff x\\
&= \int_0^\infty x^j\cdot u(x,t)''~\diff x\\
&= \underbrace{x^j\cdot u(x,t)'~\Big|_{x=0}^\infty}_{=0} - \int_0^\infty j\cdot x^{j-1}\cdot u(x,t)'~\diff x\\
&= \underbrace{- j\cdot x^{j-1}\cdot u(x,t)~\Big|_{x=0}^\infty}_{=0} + \int_0^\infty j\cdot (j-1)\cdot x^{j-2}\cdot u(x,t)~\diff x\\
&= j\cdot (j-1)\cdot s_{j-2}(t).
\end{align*}
These are the recursive relations we already encountered with the heat equation on $\rset$, see \cite{curtoHeat22,curtoHeat23}.
Therefore, for \emph{odd} polynomials $p\in\rset[x]$ we have
\[(T_tp)(x)=\int_0^\infty p(y)\cdot k_t(x,y)~\diff y = (e^{t\partial_x^2}p)(x)\in\rset[x]\]
with $T_tx^{2d+1} = \fp_{2d+1}(x,t)$ for all $d\in\nset_0$ in \cite[Dfn.\ 3.1]{curtoHeat23} and if additionally $p\geq 0$ on $[0,\infty)$, then $T_tp\geq 0$ on $[0,\infty)$.
On the other hand for \emph{even} polynomials $p\in\rset[x]$ we find that $T_tp\notin C^\infty([0,\infty))\setminus\rset[x]$ but at least $T_tp\geq 0$ on $[0,\infty)$.
The reason for this strange behavior is of course the Dirichlet boundary condition and the resulting reflection principle.
This effect needs further investigations.
\exmsymbol
\end{exm}

This example shall also serve as a warning.
Just because the operator (symbol) is $\partial_x^2$ does not mean that the positivity preservers are $\exp(t\partial_x^2)$.
The domain and the corresponding boundary conditions have to be taken into account as is done and is well-known in the partial differential equation literature and semi-group theory.
Besides the \Cref{open} below this is another direction where further studies have to be done.

\section{Summary and Open Question}

If
\begin{equation}\label{eq:Adegree}
A = \sum_{\alpha\in\nset_0^n} q_\alpha\cdot\partial^\alpha \quad\text{with}\quad q_\alpha \in \rset[x_1,\dots,x_n]_{\leq |\alpha|}
\end{equation}
then
\begin{equation}\label{eq:degreePres}
A\rset[x_1,\dots,x_n]_{\leq d}\subseteq\rset[x_1,\dots,x_n]_{\leq d} \quad\text{for all}\quad d\in\nset_0,
\end{equation}
i.e., $A$ is called \emph{degree preserving}.
In fact for any linear $A:\rset[x_1,\dots,x_n]\to\rset[x_1,\dots,x_n]$ we have (\ref{eq:Adegree}) $\Leftrightarrow$ (\ref{eq:degreePres}).
(\ref{eq:Adegree}) $\Rightarrow$ (\ref{eq:degreePres}) is clear.
For (\ref{eq:degreePres}) $\Rightarrow$ (\ref{eq:Adegree}) take in (\ref{eq:Adegree}) the smallest $\alpha$ with respect to the lex-order such that $\deg q_\alpha > |\alpha|$.
Then $\deg A x^\alpha > |\alpha|$ which is a contradiction to (\ref{eq:degreePres}).

Similar to \Cref{cor:posPDE} we have that the partial differential equation
\[\partial_t p(x,t) = Ap(x,t) \quad\text{with}\quad p(x,0) = p_0(x)\in\rset[x_1,\dots,x_n]\]
and $A$ as in (\ref{eq:Adegree})
is equivalent to a linear system of ordinary differential equations in the coefficients of $p$.
Hence, it has a unique solution $p(\,\cdot\,,t) = \exp(tA)p_0$ for all $t\in\rset$, i.e., $\exp(tA)$ is well-defined for any $A$ as in (\ref{eq:Adegree}) and $t\in\rset$.
If additionally $A$ is a positivity preserver then $\exp(tA)$ is a degree preserving positivity preserver with polynomial coefficients for all $t\geq 0$.
In \Cref{thm:mainPosGenerators} we have already seen for the positivity preservers with constant coefficients that additional $A$ appear.
The restriction that $A$ is degree preserving is a natural restriction since otherwise $\exp(tA)\rset[x_1,\dots,x_n]\not\subseteq\rset[x_1,\dots,x_n]$ for any $t>0$.
So we arrive at the following open problem.

\begin{open}\label{open}
What is the description of the set
\[\left\{A = \sum_{\alpha\in\nset_0^n} q_\alpha\cdot\partial^\alpha \;\middle|\; \exp(tA)\ \text{is a positivity preserver for all}\ t\geq 0\right\}\]
of all positivity generators, i.e., not necessarily with constant coefficients?
\end{open}

The constant coefficient case is completely solved by \Cref{thm:mainPosGenerators}.
Part of the non-constant coefficient cases are covered by the previous case that $A$ is a degree preserving positivity preserver.
For the non-constant coefficient cases note that we are then in the framework of a non-commutative infinite dimensional Lie group and its Lie algebra.
Here the theory is much richer, see e.g.\ \cite{omori74,kac85,omori97,stras02,wurz04,schmed23} and references therein.

\section*{Funding}

The author and this project are supported by the Deutsche Forschungs\-gemein\-schaft DFG with the grant DI-2780/2-1 and his research fellowship at the Zukunfs\-kolleg of the University of Konstanz, funded as part of the Excellence Strategy of the German Federal and State Government.


\providecommand{\bysame}{\leavevmode\hbox to3em{\hrulefill}\thinspace}
\providecommand{\MR}{\relax\ifhmode\unskip\space\fi MR }
\providecommand{\MRhref}[2]{%
  \href{http://www.ams.org/mathscinet-getitem?mr=#1}{#2}
}
\providecommand{\href}[2]{#2}

\end{document}